\documentclass[11pt,final]{article}
\usepackage{a4}
\usepackage{amsmath}%
\usepackage{amstext}%
\usepackage{amssymb}%
\usepackage{amsthm}
\usepackage{comment}
\usepackage{listings}

\usepackage[final]{graphicx}
\usepackage{subcaption}

\usepackage[margin=0.5in]{geometry}

\usepackage{graphicx}
\usepackage{epstopdf}

\usepackage{xcolor}

\newtheorem{theorem}{Theorem}

\newtheorem{lemma}[theorem]{Lemma}

\newtheorem{rem}[theorem]{Remark}
\newenvironment{remark}{\rem\rm}{\endrem}

\newcommand{\R}{\mathbb{R}}%
\newcommand{\N}{\mathbb{N}}%
\newcommand{\e}{\varepsilon}%
\newcommand{\ol}{\overline}%
\newcommand{\ox}{\overline{x}}

\newcommand{\n}{{\nabla}}

\newcommand{\ds}{\displaystyle}
\newcommand{\To}{\longrightarrow}
\def\a{\alpha}

\def\e{\epsilon}

\def\l{\lambda}
\def\<{\langle}
\def\>{\rangle}
\DeclareMathOperator*\prox{prox}%
\DeclareMathOperator*\argmin{argmin}

\DeclareMathOperator*\pr{pr}

\title{A Nesterov type algorithm with double Tikhonov regularization: fast convergence of the function values and strong convergence to the minimal norm solution}

\author{Mikhail Karapetyants \thanks{University of Vienna, Faculty of Mathematics, Oskar-Morgenstern-Platz 1, A-1090 Vienna, Austria,
email: mikhail.karapetyants@univie.ac.at. Research done during the visit of the  author at Technical University of Cluj-Napoca.} \and Szil\'{a}rd Csaba L\'{a}szl\'{o} \thanks{Technical University of Cluj-Napoca, Department of Mathematics, Memorandumului 28, Cluj-Napoca,
 Romania, e-mail: szilard.laszlo@math.utcluj.ro. This work was supported by a grant of the Ministry of Research, Innovation and Digitization, CNCS-UEFISCDI, project number PN-III-P1-1.1-TE-2021-0138, within PNCDI III.}}

\begin{document}
\maketitle

\noindent \textbf{Abstract.} We investigate the strong convergence properties of a Nesterov type algorithm with two Tikhonov regularization terms in connection to the minimization problem of a smooth convex function $f.$  We show that the generated sequences converge strongly to the minimal norm element from $\argmin f$. We also show that from a practical point of view the Tikhonov regularization does not affect Nesterov's optimal convergence rate of order $\mathcal{O}(n^{-2})$ for the potential energies $f(x_n)-\min f$ and $f(y_n)-\min f$, where $(x_n),\,(y_n)$ are the sequences generated by our algorithm. Further, we obtain fast convergence to zero of the discrete velocity, but also some  estimates concerning the value of the gradient of the objective function in the generated sequences.\vspace{1ex}



\section{Introduction}\label{sec-intr}

Let $\mathcal{H}$ be a Hilbert space endowed with the scalar product $\< \cdot,\cdot\>$ and norm $\|\cdot\|$ and consider the optimization problem
\begin{equation}\label{opt-pb} \inf_{x\in\mathcal{H}}f(x) \end{equation}
where $f:\mathcal{H}\To \R$ is a convex, continuously Fr\'{e}chet differentiable function, with $L$-Lipschitz continuous gradient, whose set of minimizers  $\argmin f$ is nonempty.

We associate to the optimization problem \eqref{opt-pb} the following  inertial-gradient type algorithm.
Let $x_0,x_{1}\in\mathcal{H}$
and for all $k\ge 1$ set

\begin{equation}\label{tdiscgen}
\left\{\begin{array}{lll}
y_k= x_k+b_{k-1}(x_k-x_{k-1})-c_k x_k
\\
x_{k+1}=y_k-s\n f(y_k)-s\e_k y_k.
\end{array}\right.
\end{equation}

We assume that  $s\in\left(0,\frac{1}{L}\right)$ and the sequence $(c_k)_{k\ge 1}$ is nonnegative for $k$ big enough and satisfies $\lim_{k\to+\infty}c_k=0$. Further,  we assume that $(\e_k)_{k\ge 1}$ is a non-increasing positive sequence that satisfies $\lim_{k\to+\infty}\e_k=0.$
Observe that in case the inertial parameter $(b_k)_{k\ge 0}$ satisfies $\lim_{k\to+\infty}b_k=1,$ then Algorithm \eqref{tdiscgen} has the form of the famous Nesterov algorithm, (see \cite{Nest1,CD} and also  \cite{AP,L}), with two Tikhonov regularization terms. Indeed, the terms $c_k x_k$ and $\e_k y_k$ in Algorithm \eqref{tdiscgen} play the role of Tikhonov regularization terms, consequently our aim is to obtain the strong convergence of the generated sequences to the element of minimal norm from $\argmin f,$ (see \cite{AL-siopt,abc2,ACR,att-com1996,AC,AL-nemkoz,BCL,BGMS,CPS,JM-Tikh,L-jde,L-mapr,Tikh,TA}) and at the same time to preserve the optimality of Nesterov algorithm concerning the convergence rate of order $\mathcal{O}(k^{-2})$ for the potential energy $f(x_k)-\min f$, (see \cite{Nest1,CD} ).  Our analysis reveals that the inertial parameter and the Tikhonov regularization parameters are strongly correlated. This fact is in concordance with some recent results from the literature concerning the strong convergence of the trajectories of some continuous second order dynamical systems to a minimal norm minimizer of a convex function or to the minimal norm zero of a maximally monotone operator \cite{AL-siopt,ABCR,ACR,ACR2,AL-nemkoz,BCL,BCLstr,BGMS,L-jde}.
Concerning the discrete case, that is, the case of inertial algorithms that converge strongly to the minimal norm solution of a convex optimization problem,  there are only few results in the literature, see \cite{AL-nemkoz,L-mapr} and also those refer to proximal inertial algorithms obtained via implicit discretizations of some second order continuous dynamical systems, (see \cite{SBC,L-jde}).

Indeed, in \cite{AL-nemkoz} the following inertial-proximal algorithm was considered in connection to the optimization problem \eqref{opt-pb}: $x_0,x_{1}\in\mathcal{H},\, x_{k+1}={ \rm prox}_{f}\left( x_k+\left(1-\frac{\a}{k}\right)(x_k -  x_{k-1}) - \frac{c}{k^2}x_k\right),$
where  $\a>3,\,c>0$ and $\prox\nolimits_{f} : \mathcal{H}\rightarrow \mathcal{H}, \quad \prox\nolimits_{f}(x)=\argmin_{y\in \mathcal{H}}\left(f(y)+\frac{1}{2}\|y-x\|^2\right),$ denotes the proximal point operator of the convex function $f$. Due to our best knowledge this is the first inertial algorithm  in the literature for which both strong convergence results for the generated sequences and fast convergence of the potential energy $f(x_k)-\min f$ and discrete velocity $\|x_k-x_{k-1}\|$ were obtained. However, from practical point of view, it is not natural that the minimizers of a smooth function to be approximated via proximal, i.e. backward, steps. Another drawback of this algorithm  is that does not assure the full strong convergence of the generated sequences to the minimum norm minimizer $x^*$. Indeed, according to \cite{AL-nemkoz} only the strong convergence result $\liminf_{k\to+\infty}\|x_k-x^*\|=0$ is provided. In order to overcome these deficiencies in \cite{L-mapr} the author assumed that the objective function in \eqref{opt-pb} is proper, convex and lower semicontinuous only and associated to this optimization problem the following inertial-proximal algorithm: $x_0,x_{1}\in\mathcal{H},\,x_{k+1}={ \rm prox}_{\l_k f}\left( x_k+\left(1-\frac{\a}{k^q}\right)(x_k -  x_{k-1}) - \frac{c}{k^p}x_k\right),$
where  $\a,\,q,\,c,\,p>0$ and $(\l_k)$ is a sequence of positive real numbers. According to \cite{L-mapr}, in case the stepsize $\l_k\equiv 1$ and $0<q<1,\,1<p<q+1$ the full convergence of the generated sequences to the minimum norm minimizer $x^*$ is obtained, i.e. $\lim_{k\to+\infty}\|x_k-x^*\|=0$. Further, the fast convergence of the potential energy $f(x_k)-\min f$ and discrete velocity $\|x_k-x_{k-1}\|$ were shown.

In concordance to the results emphasized above, the main goal of this paper is to obtain similar results for gradient type inertial algorithms. Unfortunately our parameters in Algorithm \eqref{tdiscgen} will not have such simple forms as the parameters in \cite{AL-nemkoz} or \cite{L-mapr} and this is due to the fact that we cannot use a discrete Lyapunov function of similar form as the ones considered in \cite{AL-nemkoz,L-mapr}, instead we have to construct a new discrete Lyapunov function suitable for our analysis. Therefore, the forms of the given sequences $(b_k)_{k\ge 0},\,(c_k)_{k\ge 1}$ are crucial in order to obtain our results. More precisely, for given Tikhonov regularization parameter $(\e_k)_{k\ge 1}$ and a fixed stepsize $s$ consider the sequence $(q_k)_{k\ge 0}$ which after an index $k$ big enough, satisfies
$$(Q)\,\,\,\,(1-s\e_{k+1})^2q_{k+1}^2-(1-s\e_k)^2q_k^2-2sq_{k+1}+s(1-s\e_k)^2q_k\le0,\, q_{k}\ge \frac{2s}{(1-s\e_k)^2}.$$
 Then, the inertial parameter $(b_k)_{k\ge 0}$ and the regularization parameter $(c_k)_{k\ge 1}$ from Algorithm \eqref{tdiscgen} are defined via the conditions
$$(B)\,\,\,\left\{\begin{array}{lll}
  b_{k-1}=0, \mbox{ if }k=1\mbox{ or } (1-s\e_{k-1})(1-s\e_{k})q_{k-1}q_k=0\\
  \\
 \ds b_{k-1}=\frac{(q_{k-1}-s)((1-s\e_{k-1})^2q_{k-1}-2s)}{(1-s\e_{k-1})(1-s\e_{k})q_{k-1}q_k},\mbox{ otherwise}
\end{array}\right.$$
and
 $$(C)\,\,\,\left\{\begin{array}{lll}
  c_{k}=0, \mbox{ if }k=1\mbox{ or } (1-s\e_{k-1})(1-s\e_{k})q_{k-1}q_k=0\\
  \\
 \ds c_{k}=\frac{2s}{(1-s\e_{k-1})(1-s\e_{k})^2q_k}\left(\frac{s}{q_{k-1}}-\frac{s^2\e_k}{q_{k-1}}-s(\e_{k-1}-\e_k)\right),\mbox{ otherwise}.
\end{array}\right.$$
Note that despite of the complex form of these parameters, from a practical perspective, Algorithm \eqref{tdiscgen} can easily be implemented.

A comprehensive analysis of the above conditions will be carried out in section 4. Here we just underline that in case we specify the parameters as $\e_k=\frac{c}{k^p},\,c,p>0$ and we take $q_k=ak^q,\,a>0,\, 0<q<1$ then (Q) is satisfied for every fixed stepsize $s\in\left(0,\frac{1}{L}\right)$ and the main result of the paper can be summarized in the following theorem.

\begin{theorem}
For  $ p < 2q$  let $(x_k)_{k\ge 0}$ and $(y_k)_{k\ge 1}$ be the sequences generated by Algorithm \eqref{tdiscgen}. Then, $(x_k)$ and $(y_k)$ converge strongly to $x^*$, where $\{x^*\}=\pr_{\argmin f}(0)$ is the minimum norm minimizer of our objective function $f.$
Further, $f(x_k)-\min f=\mathcal{O}\left(k^{-p}\right),\mbox{ as }k\to+\infty,\mbox{ and }f(y_k)-\min f=\mathcal{O}\left(k^{-p}\right),\mbox{ as }k\to+\infty.$ Additionally, $ \|\n f(x_k)\| = \ o\left( k^{-\frac{p}{2}} \right) \mbox{ as } k \to +\infty,\,\|\n f(y_k)\| = \ o\left( k^{-\frac{p}{2}} \right) \mbox{ as } k \to +\infty\mbox{ and }\| x_k - x_{k-1} \| \ = \ o\left( k^{-\frac{p}{2}} \right) \mbox{ as } k \to +\infty.$
\end{theorem}

The paper is organized as follows. In the next section we present some preliminary results and notions that we need to carry out our analysis. In section 3 we prove the main result of the paper. We obtain strong convergence of the sequences generated by Algorithm \eqref{tdiscgen} and also fast convergence of the potential energy and discrete velocity. In section 4 we consider the parameters in a simple form and discuss the conditions these parameters must satisfy in order to obtain the results presented at section 3. Further, in section 5 via some numerical experiments we show that Algorithm \eqref{tdiscgen} indeed assures the convergence of the generated sequences to a minimal norm solution and also that both Tikhonov regularization terms are indispensable in order to obtain this result. Finally, we conclude our paper with some future research plans.

\section{Preliminary results}
In order to obtain strong convergence   for the sequence $(x_k)$ generated by  Algorithm \eqref{tdiscgen} we need some preliminary results. The first one is the Descent Lemma \cite{Nest}.
\begin{lemma}\label{descent} Let $f:\mathcal{H}\To\R$ be a  smooth function, with $L-$Lipschitz continuous gradient. Then,
$$f(x)\le f(y)+\<\n f(y),x-y\>+\frac{L}{2}\|y-x\|^2,\,\forall x,y\in\mathcal{H}.$$
\end{lemma}

Further, we need the following property of smooth, convex functions,  see \cite{Nest}.

\begin{lemma}\label{gradLineq} Let $f:\mathcal{H}\To\R$ be a convex smooth function, with $L-$Lipschitz continuous gradient. Then,
$$\frac{1}{2L}\|\n f(y)-\n f(x)\|^2+\<\n f(y),x-y\>+f(y)\le f(x),\mbox{ for all }x,y\in\mathcal{H}.$$
\end{lemma}

Our first original result is a modified descent lemma, which in particular contains Lemma 1 from \cite{ACFR}, however has a considerable simplified proof.

\begin{lemma}\label{moddesc} Let $f:\mathcal{H}\To\R$ be a convex smooth function, with $L-$Lipschitz continuous gradient and let $s>0.$ Then,
\begin{equation}\label{f3}
f(y-s\n f(y))\le f(x)+\<\n f(y),y-x\>+\left(\frac{L}{2}s^2-s\right)\|\n f(y)\|^2-\frac{1}{2L}\|\n f(y)-\n f(x)\|^2,\,\forall x,y\in\mathcal{H}.
\end{equation}
Assume further that $s\in\left(0,\frac{1}{L}\right].$ Then,
\begin{equation}\label{f4}
f(y-s\n f(y))\le f(x)+\<\n f(y),y-x\>-\frac{s}{2}\|\n f(y)\|^2-\frac{s}{2}\|\n f(y)-\n f(x)\|^2,\,\forall x,y\in\mathcal{H}.
\end{equation}
\end{lemma}
\begin{proof}
Indeed, by taking $x=y-s\n f(y)$ in Lemma \ref{descent}, we get
\begin{equation}\label{f1}
f(y-s\n f(y))\le f(y)+\left(\frac{L}{2}s^2-s\right)\|\n f(y)\|^2,\,\forall y\in\mathcal{H}.
\end{equation}
From Lemma \ref{gradLineq}  we have
\begin{equation}\label{f2}
f(y)\le f(x)+\<\n f(y),y-x\>-\frac{1}{2L}\|\n f(y)-\n f(x)\|^2,\,\forall x,y\in\mathcal{H}.
\end{equation}

Combining \eqref{f1} and \eqref{f2} we get
\begin{equation}\nonumber
f(y-s\n f(y))\le f(x)+\<\n f(y),y-x\>+\left(\frac{L}{2}s^2-s\right)\|\n f(y)\|^2-\frac{1}{2L}\|\n f(y)-\n f(x)\|^2,\,\forall x,y\in\mathcal{H},
\end{equation}
which is nothing else that \eqref{f3}.

Assume that $0<s\le\frac{1}{L}.$ Then,
$$\frac{L}{2}s^2-s\le-\frac{s}{2}
\mbox{ and }-\frac{1}{2L}\le-\frac{s}{2},$$
hence \eqref{f3} leads to \eqref{f4}, that is
\begin{equation}\nonumber
f(y-s\n f(y))\le f(x)+\<\n f(y),y-x\>-\frac{s}{2}\|\n f(y)\|^2-\frac{s}{2}\|\n f(y)-\n f(x)\|^2,\,\forall x,y\in\mathcal{H}.
\end{equation}
\end{proof}

We continue the present section by emphasizing the main idea behind the Tikhonov regularization, which will assure strong convergence results for the sequence generated our algorithm \eqref{tdiscgen} to a minimizer of minimal norm of the objective function $f$. By  $\ox_{k}$ we denote the unique solution of the strongly convex minimization problem
\begin{align*}
 \min_{x \in \mathcal{H}} \left( f(x) + \frac{\e_k}{2} \| x \|^2 \right).
\end{align*}
We know, (see for instance \cite{att-com1996}), that $\lim\limits_{k \to +\infty} \ox_{k}=x^\ast$, where $x^\ast = \argmin\limits_{x \in \argmin f} \| x \|$ is the minimal norm element from the set $\argmin f.$ Obviously, $\{x^*\}=\pr_{\argmin f} 0$ and we have the inequality $\| \ox_{k} \| \leq \| x^\ast \|$ (see \cite{BCL}).

Since $\ox_{k}$ is the unique minimizer of the strongly convex function $f_k(x)=f(x)+\frac{\e_k}{2}\|x\|^2,$ obviously one has
\begin{equation}\label{fontos0}
\n f_k(\ox_{k})=\n f(\ox_{k})+\e_k\ox_{k}=0.
\end{equation}
Further, from Lemma A.1 c) from \cite{L-mapr} we have
\begin{equation}\label{lfos}
\left\|\ox_{k+1}-\ox_k\right\|\le\min\left(\frac{\e_k-\e_{k+1}}{\e_{k+1}}\|\ox_{k}\|,\frac{\e_k-\e_{k+1}}{\e_k}\|\ox_{k+1}\|\right).
\end{equation}

Note that since $f_k$ is strongly convex, from the gradient inequality we have
\begin{equation}\label{fontos2}
f_k(y)-f_k(x)\ge\<\n f_k(x),y-x\>+\frac{\e_k}{2}\|x-y\|^2,\mbox{ for all }x,y\in\mathcal{H}.
\end{equation}
In particular
\begin{equation}\label{fontos3}
f_k(x)-f_k(\ox_k)\ge\frac{\e_k}{2}\|x-\ox_k\|^2,\mbox{ for all }x\in\mathcal{H}.
\end{equation}

Moreover, observe that for all $x,y\in\mathcal{H}$, one has

\begin{equation}\label{fontos5}
f(x)-f(y)=(f_k(x)-f_k(\ox_k))+(f_k(\ox_k)-f_k(y))+\frac{\e_k}{2}(\|y\|^2-\|x\|^2)\le f_k(x)-f_k(\ox_k)+\frac{\e_k}{2}\|y\|^2.
\end{equation}

Note that $\n f_k(x)=\n f(x)+\e_kx$, consequently if $\n f$ is L-Lipschitz continuous then the Lipschitz constant of the gradient of $f_k$ is $L+\e_k.$

Hence, if we apply Lemma \ref{moddesc}  to $f_k$ we get that for all $s\in\left(0,\frac{1}{L+\e_k}\right]$ one has
\begin{equation}\label{tf4}
f_k(y-s\n f_k(y))\le f_k(x)+\<\n f_k(y),y-x\>-\frac{s}{2}\|\n f_k(y)\|^2-\frac{s}{2}\|\n f_k(y)-\n f_k(x)\|^2,\,\forall x,y\in\mathcal{H}.
\end{equation}

Now, we can rewrite Algorithm \eqref{tdiscgen} in a more convenable equivalent form, by using the strongly convex function $f_k.$ Indeed, since $\n f_k(x)=\n f(x)+\e_kx$, Algorithm \eqref{tdiscgen} can equivalently be written as:
$x_0,x_1\in\mathcal{H}$ and for all $k\ge 1$
\begin{equation}\label{tdiscgen1}
\left\{\begin{array}{lll}
y_k= x_k+b_{k-1}(x_k-x_{k-1})-c_k x_k
\\
x_{k+1}=y_k-s\n f_k(y_k).
\end{array}\right.
\end{equation}

\section{Strong convergence}

In this section we provide sufficient conditions such that the sequences generated by \eqref{tdiscgen1} converge strongly to the minimum norm minimizer of $f$ and at the same time  fast convergence of the function values in the generated sequences and also fast convergence of the discrete velocity to zero are obtained. Moreover, we also show some pointwise estimates for the gradient of the objective function.

In order to obtain our  general result concerning the strong  convergence of the sequences generated by the algorithm \eqref{tdiscgen1} we need  to use \eqref{tf4}, hence we adjust the indexes in algorithm \eqref{tdiscgen1} as follows.

Assume that $0<s<\frac{1}{L}$ and  let $ k_0\in\N$ such that the following assumption holds.
$$(S)\,\,\,\,s\in \left(0,\frac{1}{L+\e_{k_0}}\right]\subseteq \left(0,\frac{1}{L+\e_k}\right],\mbox{ for all }k\ge k_0.$$
Note that such index $k_0$ exists, since $\e_k$ is nonincreasing and $\e_k\to 0$ as $k\to+\infty.$\\
Further, since $\e_k$ is nonincreasing there exists $k_1\ge k_0$ such that $1-s\e_k>0$ for all $k\ge k_1.$

Consider the sequence $(q_k)_{k\ge 0}$ which after an index $k_2\ge k_1$, satisfies the condition (Q), that is
$$(Q)\,\,\,\,(1-s\e_{k+1})^2q_{k+1}^2-(1-s\e_k)^2q_k^2-2sq_{k+1}+s(1-s\e_k)^2q_k\le0,\, q_{k}\ge \frac{2s}{(1-s\e_k)^2}$$ for all $k\ge k_2.$
Note that $q_k\ge 2s>0$ for all $k\ge k_2.$

Let $\ol k=k_2+1$ and observe that $(1-s\e_{k-1})(1-s\e_{k})q_{k-1}q_k>0$ for all $k\ge \ol k$, consequently the sequences $b_k$ and $c_k$ defined at (B) and (C) have the following forms:
$$b_{k-1}=\frac{(q_{k-1}-s)((1-s\e_{k-1})^2q_{k-1}-2s)}{(1-s\e_{k-1})(1-s\e_{k})q_{k-1}q_k},\mbox{ for all } k\ge \ol k$$
and
$$c_{k}=\frac{2s}{(1-s\e_{k-1})(1-s\e_{k})^2q_k}\left(\frac{s}{q_{k-1}}-\frac{s^2\e_k}{q_{k-1}}-s(\e_{k-1}-\e_k)\right),\mbox{ for all } k\ge \ol k.$$

The following general result holds.

\begin{theorem}\label{strongconvergence}  For a sequence $(q_k)$ satisfying (Q) and the stepsize $s$ satisfying (S), consider the sequences $(b_k)_{k\ge 0}$ and $(c_k)_{k\ge 1}$  defined at (B) and (C) and  let $(x_k)_{k\ge 0},\,(y_k)_{k\ge 1}$ be the sequences generated by Algorithm \eqref{tdiscgen}.
Assume that  the sequence $\left(\frac{q_k\e_k}{q_{k-1}\e_{k-1}}\right)_{k\ge\ol k}$ is bounded, the sequence $(q_k^2\e_k)_{k\ge\ol k-1}$ is increasing, further $\lim_{k\to+\infty}q_k^2\e_k=+\infty$ and
$\lim_{k\to+\infty}\frac{ q_k(\e_k-\e_{k+1})}{\e_k} =0$.

Then, $(x_k)$ converges strongly to $x^*$, where $\{x^*\}=\pr_{\argmin f}(0)$ is the minimum norm minimizer of our objective function $f.$ Moreover  $\| x_k - y_k \| \ = \ o\left( \sqrt{\e_k} \right) \mbox{ as } k \to +\infty,$ hence $(y_k)$ also converges strongly to $x^*.$

Further, the following estimates hold.
\[
f_k(x_{k})-f_{k}(\ol x_{k})=o(\e_k)\mbox{ as }k\to+\infty,
\]
\[
f(x_k)-\min f=\mathcal{O}\left(\e_k\right),\mbox{ as }k\to+\infty\mbox{ and }f(y_k)-\min f=\mathcal{O}\left(\e_k\right),\mbox{ as }k\to+\infty,
\]
\[
\| x_k - x_{k-1} \| \ = \ o\left( \sqrt{\e_k} \right) \mbox{ as } k \to +\infty,
\]
and
\[
\| \n f(x_k) \| = \ o\left( \sqrt{\e_k} \right) \mbox{ as } k \to +\infty\mbox{ and }\| \n f(y_k) \| = \ o\left( \sqrt{\e_k} \right) \mbox{ as } k \to +\infty.
\]
\end{theorem}
\begin{proof}
Assume  that $k\ge\ol k$. We take $y=y_k,\,x=x_k$ in \eqref{tf4}  and we get
\begin{equation}\label{trightforp}
f_k(x_{k+1})\le f_k(x_k)+\<\n f_k(y_k),y_k-x_k\>-\frac{s}{2}\|\n f_k(y_k)\|^2-\frac{s}{2}\|\n f_k(y_k)-\n f_k(x_k)\|^2,\,\forall k\ge \ol k.
\end{equation}

Now we take $y=y_k,\,x=x^*$ in \eqref{tf4}  and taking into account that $\n f(x^*)=0$ we get
\begin{equation}\label{trightforq}
f_k(x_{k+1})\le f_k(x^*)+\<\n f_k(y_k),y_k-x^*\>-\frac{s}{2}\|\n f_k(y_k)\|^2-\frac{s}{2}\|\n f_k(y_k)-\e_k x^*\|^2,\,\forall k\ge \ol k.
\end{equation}

Consider the sequence $(p_k)_{k\ge \ol k}$  defined by
 \begin{equation}\label{trightforpq}
 p_k=\frac{(1-s\e_k)^2q_k^2}{2s}-q_k,
\end{equation}
for all $k\ge \ol k.$ Note that due to assumption $(Q)$ one has $p_k\ge 0$ for all $k\ge \ol k.$

 We multiply \eqref{trightforp} with $p_k$ and \eqref{trightforq} with $q_k$ and add to get
\begin{align}\label{trightfirst}
(p_k+q_k)(f_k(x_{k+1})-f(x^*))-&p_k(f_k(x_{k})-f(x^*))\le\\
\nonumber& \left\<\n f_k(y_k),(p_k+q_k)y_k-p_k x_k+(s\e_k-1)q_k x^*-\frac{s}{2}(p_k+q_k)\n f_k(y_k)\right\>\\
\nonumber &-\frac{s}{2}p_k\|\n f_k(y_k)-\n f_k(x_k)\|^2-\frac{s}{2}q_k\|\n f_k(y_k)\|^2-\frac{q_k}{2}\e_k(s\e_k-1)\|x^*\|^2,
\end{align}
for all $k\ge \ol k.$

Now by neglecting the nonpositive terms from the right hand side of \eqref{trightfirst} we obtain

\begin{align}\label{trightfirst1}
(p_k+q_k)(f_k(x_{k+1})-f(x^*))-&p_k(f_k(x_{k})-f(x^*))-q_k\frac{\e_k}{2}\|x^*\|^2\le\\
\nonumber& \left\<\n f_k(y_k),(p_k+q_k)y_k-p_k x_k+(s\e_k-1)q_k x^*-\frac{s}{2}(p_k+q_k)\n f_k(y_k)\right\>,
\end{align}
for all $k\ge \ol k.$

Further, by using the fact that $(\e_k)$ is nonincreasing we have $$f_k(x_{k+1})=f_{k+1}(x_{k+1})+\frac{\e_k-\e_{k+1}}{2}\|x_{k+1}\|^2\ge f_{k+1}(x_{k+1}),$$ consequently it holds
\begin{align}\label{righttra}
&(p_k+q_k)(f_k(x_{k+1})-f(x^*))-p_k(f_k(x_{k})-f(x^*))-q_k\frac{\e_k}{2}\|x^*\|^2=(p_k+q_k)f_k(x_{k+1})-p_k f_k(x_{k})\\
\nonumber&-q_kf_k(x^*)=(p_k+q_k)(f_k(x_{k+1})-f_{k+1}(\ol x_{k+1})) +(p_k+q_k)f_{k+1}(\ol x_{k+1})-p_k(f_k(x_{k})-f_k(\ol x_{k}))\\
\nonumber&-p_k f_k(\ol x_{k})-q_kf_k(x^*)\ge (p_k+q_k)(f_{k+1}(x_{k+1})-f_{k+1}(\ol x_{k+1}))-(p_{k-1}+q_{k-1})(f_k(x_{k})-f_k(\ol x_{k}))\\
\nonumber&+(p_{k-1}+q_{k-1}-p_k)(f_k(x_{k})-f_k(\ol x_{k}))+p_k(f_{k+1}(\ol x_{k+1})-f_{k}(\ol x_{k}))+q_k(f_{k+1}(\ol x_{k+1})-f_k(x^*)).
\end{align}
In one hand, according to \eqref{fontos3} one has $f_{k}(\ox_{k+1})-f_{k}(\ox_{k})\ge\frac{\e_{k}}{2}\|\ox_{k+1}-\ox_k\|^2$ hence
\begin{align}\label{righttrala}
&p_k(f_{k+1}(\ol x_{k+1})-f_{k}(\ol x_{k}))=p_k\left(f_{k}(\ol x_{k+1})-f_{k}(\ol x_{k})+\frac{\e_{k+1}-\e_k}{2}\|\ol x_{k+1}\|^2\right)\\
\nonumber&\ge p_k\frac{\e_{k}}{2}\|\ox_{k+1}-\ox_k\|^2+p_k\frac{\e_{k+1}-\e_k}{2}\|\ol x_{k+1}\|^2\ge p_k\frac{\e_{k+1}-\e_k}{2}\|\ol x_{k+1}\|^2.
\end{align}
On the other hand, by using the gradient inequality we get
\begin{align}\label{righttralala}
&q_k(f_{k+1}(\ol x_{k+1})-f_k(x^*))=q_k\left(f_{k}(\ol x_{k+1})-f_k(x^*)+\frac{\e_{k+1}-\e_k}{2}\|\ol x_{k+1}\|^2\right)\\
\nonumber&\ge q_k\e_k \<x^*,\ol x_{k+1}-x^*\>+q_k\frac{\e_{k+1}-\e_k}{2}\|\ol x_{k+1}\|^2.
\end{align}
 Hence, combining \eqref{trightfirst1}, \eqref{righttra}, \eqref{righttrala} and \eqref{righttralala} we obtain
 \begin{align}\label{trightsecond}
&(p_k+q_k)(f_{k+1}(x_{k+1})-f_{k+1}(\ol x_{k+1}))-(p_{k-1}+q_{k-1})(f_k(x_{k})-f_k(\ol x_{k}))\\
\nonumber&+(p_{k-1}+q_{k-1}-p_k)(f_k(x_{k})-f_k(\ol x_{k}))\le q_k\e_k \<x^*,x^*-\ol x_{k+1}\>+(p_k+q_k)\frac{\e_k-\e_{k+1}}{2}\|\ol x_{k+1}\|^2\\
\nonumber&+\left\<\n f_k(y_k),(p_k+q_k)y_k-p_k x_k+(s\e_k-1)q_k x^*-\frac{s}{2}(p_k+q_k)\n f_k(y_k)\right\>,\mbox{ for all }k\ge\ol k.
\end{align}

Now, according to the form of $p_k$ and condition $(Q)$ one has
$$p_{k-1}+q_{k-1}-p_k=\frac{(1-s\e_{k-1})^2q_{k-1}^2}{2s}-\frac{(1-s\e_k)^2q_k^2}{2s}+q_k\ge s\frac{(1-s\e_{k-1})^2}{2s}q_{k-1}=(p_{k-1}+q_{k-1})\frac{s}{q_{k-1}},$$
for all $k\ge\ol k.$
Consequently, \eqref{trightsecond} leads to
\begin{align}\label{trightthird}
&(p_k+q_k)(f_{k+1}(x_{k+1})-f_{k+1}(\ol x_{k+1}))-(p_{k-1}+q_{k-1})(f_k(x_{k})-f_k(\ol x_{k}))\\
\nonumber&+(p_{k-1}+q_{k-1})\frac{s}{q_{k-1}}(f_k(x_{k})-f_k(\ol x_{k}))\le q_k\e_k \<x^*,x^*-\ol x_{k+1}\>+(p_k+q_k)\frac{\e_k-\e_{k+1}}{2}\|\ol x_{k+1}\|^2\\
\nonumber&+\left\<\n f_k(y_k),(p_k+q_k)y_k-p_k x_k+(s\e_k-1)q_k x^*-\frac{s}{2}(p_k+q_k)\n f_k(y_k)\right\>,\mbox{ for all }k\ge\ol k.
\end{align}

For all $k\ge \ol k$, consider now the sequence
$$\eta_k=\frac{(p_k+q_k)y_k-p_k x_k}{(1-\frac{s}{q_{k-1}})(1-s\e_k)q_k}=\frac{(1-s\e_k)q_k}{2s(1-\frac{s}{q_{k-1}})}y_k-\left(\frac{(1-s\e_k)q_k}{2s(1-\frac{s}{q_{k-1}})}-\frac{1}{(1-\frac{s}{q_{k-1}})(1-s\e_k)}\right)x_k.$$
Recall that  $y_k=x_k+b_{k-1}(x_k-x_{k-1})-c_kx_k$,
hence, by using Algorithm \eqref{tdiscgen1} one has
\begin{align}\label{cckek}
\nonumber\eta_{k+1}&=\frac{(1-s\e_{k+1})q_{k+1}((1+b_k-c_{k+1})x_{k+1}-b_k x_k)}{2s(1-\frac{s}{q_{k}})}-\left(\frac{(1-s\e_{k+1})q_{k+1}}{2s(1-\frac{s}{q_{k}})}-\frac{1}{(1-\frac{s}{q_{k}})(1-s\e_{k+1})}\right)x_{k+1}\\
\nonumber&=\left(\frac{(1-s\e_{k+1})q_{k+1}(b_k-c_{k+1})}{2s(1-\frac{s}{q_{k}})}+\frac{1}{(1-\frac{s}{q_{k}})(1-s\e_{k+1})}\right)y_k-\frac{(1-s\e_{k+1})q_{k+1}b_k}
{2s(1-\frac{s}{q_{k}})}x_k\\
\nonumber&-\left(\frac{(1-s\e_{k+1})q_{k+1}(b_k-c_{k+1})}{2s(1-\frac{s}{q_{k}})}+\frac{1}{(1-\frac{s}{q_{k}})(1-s\e_{k+1})}\right)s\n f_k(y_k)\\
&=\left(1-\frac{s}{q_{k-1}}\right)\eta_k-\frac{(1-s\e_k)q_k}{2}\n f_k(y_k),\mbox{ for all }k\ge \ol k.
\end{align}

In what follows we show that
\begin{align}\label{righttelso}
&\left\<\n f_k(y_k),(p_k+q_k)y_k-p_k x_k+(s\e_k-1)q_k x^*-\frac{s}{2}(p_k+q_k)\n f_k(y_k)\right\>\le\\
\nonumber&\left(1-\frac{s}{q_{k-1}}\right)\|\eta_k-x^*\|^2-\|\eta_{k+1}-x^*\|^2+\frac{s}{q_{k-1}}\|x^*\|^2,\mbox{ for all }k\ge\ol k.
\end{align}

Indeed, by using \eqref{cckek} we get
\begin{align*}
&\left(1-\frac{s}{q_{k-1}}\right)\|\eta_k-x^*\|^2-\|\eta_{k+1}-x^*\|^2+\frac{s}{q_{k-1}}\|x^*\|^2\\
&=\left(1-\frac{s}{q_{k-1}}\right)\|\eta_k\|^2-\|\eta_{k+1}\|^2-2\left\<\left(1-\frac{s}{q_{k-1}}\right)\eta_k-\eta_{k+1},x^*\right\>\\
&=\left(1-\frac{s}{q_{k-1}}\right)\frac{s}{q_{k-1}}\|\eta_k\|^2+\left\<\left(1-\frac{s}{q_{k-1}}\right)\eta_k,(1-s\e_k)q_k\n f_k(y_k)\right\>-\frac{(1-s\e_k)^2 q_k^2}{4}\|\n f_k(y_k)\|^2\\
&-\<(1-s\e_k)q_k \n f_k(y_k),x^*\>=\left\<\n f_k(y_k),(p_k+q_k)y_k-p_k x_k+(s\e_k-1)q_k x^*-\frac{s}{2}(p_k+q_k)\n f_k(y_k)\right\>\\
&+\left(1-\frac{s}{q_{k-1}}\right)\frac{s}{q_{k-1}}\|\eta_k\|^2.
\end{align*}

Consequently, by denoting $E_{k}=(p_k+q_k)(f_{k+1}(x_{k+1})-f_{k+1}(\ol x_{k+1}))+\|\eta_{k+1}-x^*\|^2$, \eqref{trightthird} and \eqref{righttelso} lead to
\begin{align}\label{trightfourth}
&E_k-E_{k-1}+\frac{s}{q_{k-1}}E_{k-1}\le \frac{s}{q_{k-1}}\|x^*\|^2+q_k\e_k \<x^*,x^*-\ol x_{k+1}\>+(p_k+q_k)\frac{\e_k-\e_{k+1}}{2}\|\ol x_{k+1}\|^2,
\end{align}
 for all $k\ge\ol k.$

Consider now the sequence  $\pi_k=\frac{1}{\prod_{i=\ol k}^k \left(1-\frac{s}{q_{i-1}}\right)}.$ Note that $(\pi_k)_{k\ge\ol k}$ is well defined, positive and increasing since by the hypotheses we have $q_{k-1}\ge 2s$ for all $k\ge \ol k.$   Further, since $\frac{1}{1-\frac{s}{q_{i-1}}}\le 2$ we have that $\pi_k\le 2^{k-\ol k+1}.$

Now, by multiplying \eqref{trightfourth} with $\pi_k$ we obtain
 \begin{align}\label{trightfifth}
\pi_k E_k-\pi_{k-1}E_{k-1}&\le \frac{s}{q_{k-1}}\pi_k\|x^*\|^2+q_k\e_k\pi_k \<x^*,x^*-\ol x_{k+1}\>+(p_k+q_k)\frac{\e_k-\e_{k+1}}{2}\pi_k\|\ol x_{k+1}\|^2\\
\nonumber&=(\pi_k-\pi_{k-1})\|x^*\|^2+q_k\e_k\pi_k \<x^*,x^*-\ol x_{k+1}\>+(p_k+q_k)\frac{\e_k-\e_{k+1}}{2}\pi_k\|\ol x_{k+1}\|^2,
\end{align}
for all $k>\ol k.$

By summing up \eqref{trightfifth} from $k=\ol k+1$ to $k=n>\ol k+1$ we obtain
 \begin{align}\label{trightsixth}
\pi_n E_n&\le\pi_n\|x^*\|^2+\sum_{k=\ol k+1}^n q_k\e_k\pi_k \<x^*,x^*-\ol x_{k+1}\>+\sum_{k=\ol k+1}^n(p_k+q_k)\frac{\e_k-\e_{k+1}}{2}\pi_k\|\ol x_{k+1}\|^2+C,
\end{align}
for some $C>0.$

 Next we show that
$$\frac{\pi_n\|x^*\|^2+\sum_{k=\ol k+1}^n q_k\e_k\pi_k \<x^*,x^*-\ol x_{k+1}\>+\sum_{k=\ol k+1}^n(p_k+q_k)\frac{\e_k-\e_{k+1}}{2}\pi_k\|\ol x_{k+1}\|^2+C}{\pi_n}=o(q_n^2\e_n)\mbox{ as }n\to+\infty.$$

Indeed, according to the hypotheses $q_n^2\e_n\to+\infty$ as $n\to+\infty$ and we know that $(\pi_n)$ is increasing, hence $\frac{\pi_n\|x^*\|+C}{\pi_n}=o(q_n^2\e_n)\mbox{ as }n\to+\infty.$

Further, since  $\left(\frac{q_k\e_k}{q_{k-1}\e_{k-1}}\right)$ is bounded, $(q_k^2\e_k\pi_k)$ is increasing  and $\lim_{n\to+\infty}q_n^2\e_n\pi_n=+\infty$,  by using the fact that $\lim_{n\to+\infty}\<x^*,x^*-\ol x_{n+1}\>=0$, via the Ces\`aro-Stolz theorem  we get
\begin{align*}\lim_{n\to+\infty}\frac{\sum_{k=\ol k+1}^n q_k\e_k\pi_k \<x^*,x^*-\ol x_{k+1}\>}{q_n^2\e_n\pi_n}&=\lim_{n\to+\infty}\frac{q_n\e_n\pi_n \<x^*,x^*-\ol x_{n+1}\>}{q_n^2\e_n\pi_n-q_{n-1}^2\e_{n-1}\pi_{n-1}}=\lim_{n\to+\infty}\frac{ \<x^*,x^*-\ol x_{n+1}\>}{q_n-\frac{q_{n-1}^2\e_{n-1}\pi_{n-1}}{q_n\e_n\pi_n}}\\
&=\lim_{n\to+\infty}\frac{ \<x^*,x^*-\ol x_{n+1}\>}{\frac{q_n^2\e_n-q_{n-1}^2\e_{n-1}}{q_n\e_n}+s\frac{q_{n-1}\e_{n-1}}{q_n\e_n}}\le \lim_{n\to+\infty}\frac{ \<x^*,x^*-\ol x_{n+1}\>}{s\frac{q_{n-1}\e_{n-1}}{q_n\e_n}} =0.
\end{align*}

Finally, according to the hypotheses $\lim_{n\to+\infty}\frac{ q_n(\e_n-\e_{n+1})}{\e_n} =0$, hence for some $M>0$ one has
\begin{align*}\lim_{n\to+\infty}\frac{\sum_{k=\ol k+1}^n(p_k+q_k)\frac{\e_k-\e_{k+1}}{2}\pi_k\|\ol x_{k+1}\|^2}{q_n^2\e_n\pi_n}&=
\frac{1}{4s}\lim_{n\to+\infty}\frac{(1-s\e_n)^2q_n^2(\e_n-\e_{n+1})\pi_n\|\ol x_{n+1}\|^2}{q_n^2\e_n\pi_n-q_{n-1}^2\e_{n-1}\pi_{n-1}}=\\
&=\frac{1}{4s}\lim_{n\to+\infty}\frac{ (1-s\e_n)^2(\e_n-\e_{n+1})\|\ol x_{n+1}\|^2}{\e_n-\frac{q_{n-1}^2\e_{n-1}\pi_{n-1}}{q_n^2\pi_n}}\\
&=\frac{1}{4s}\lim_{n\to+\infty}\frac{(1-s\e_n)^2(\e_n-\e_{n+1})\|\ol x_{n+1}\|^2}{\frac{q_n^2\e_n-q_{n-1}^2\e_{n-1}}{q_n^2}+s\frac{q_{n-1}\e_{n-1}}{q_n^2}}\\
&\le M\lim_{n\to+\infty}\frac{ q_n(\e_n-\e_{n+1})}{\e_n} =0.
\end{align*}
Consequently, from \eqref{trightsixth} we get
$E_n=o(q_n^2\e_n)\mbox{ as }n\to+\infty$,
which, taking into account the form of $E_n$, leads to
$(p_n+q_n)(f_{n+1}(x_{n+1})-f_{n+1}(\ol x_{n+1}))=o(q_n^2\e_n)\mbox{ as }n\to+\infty$. In other words,
$f_{n}(x_{n})-f_{n}(\ol x_{n})=o(\e_n)\mbox{ as }n\to+\infty$ and by using \eqref{fontos5} we get
$$f(x_n)-\min f=\mathcal{O}(\e_n)\mbox{ as }n\to+\infty.$$

In order to show strong convergence, we use \eqref{fontos3} and we get
\begin{align*}\lim_{n\to+\infty}\|x_n-x^*\|^2&\le 2\lim_{n\to+\infty}\left(\|x_n-\ol x_n\|^2+\|\ol x_n-x^*\|^2\right)\\
&\le4 \lim_{n\to+\infty}\frac{f_n(x_n)-f_n(\ol x_n)}{\e_n}+2\lim_{n\to+\infty}\|\ol x_n-x^*\|^2=0.
\end{align*}

Concerning the rates of convergence for the discrete velocity $\| x_n - x_{n-1} \|$ we conclude the following. From the definition of $E_n$ and the fact that $E_n \ = \ o\left( q_n^2 \e_n \right) \mbox{ as } n \to +\infty$
we have that
$$\| \eta_n - x^* \| \ = \ o\left( q_n \sqrt{\e_n} \right) \mbox{ as } n \to +\infty.$$
Now, using the definition of $\eta_n$ and the fact that $y_n = x_n+b_{n-1}(x_n-x_{n-1})-c_nx_n$ we derive
\begin{align*}
    \eta_n -x^*\ &= \frac{(1-s\e_n)q_n}{2s(1-\frac{s}{q_{n-1}})}y_n-\left(\frac{(1-s\e_n)q_n}{2s(1-\frac{s}{q_{n-1}})}-\frac{1}{(1-\frac{s}{q_{n-1}})(1-s\e_n)}\right)x_n -x^*\\
    &= \frac{(1-s\e_n)q_n}{2s(1-\frac{s}{q_{n-1}})}b_{n-1} (x_n - x_{n-1}) + \left(\frac{1}{(1-\frac{s}{q_{n-1}})(1-s\e_n)}-\frac{(1-s\e_n)q_n}{2s(1-\frac{s}{q_{n-1}})}c_n\right)x_n-x^*
\end{align*}
Now, since $(q_n^2\e_n)$ is increasing, and $(\e_n)$ is non-increasing we deduce that $(q_n)$ is increasing. Further, since $\lim_{n\to+\infty} q_n^2\e_n=+\infty$ and $\lim_{n\to+\infty} \e_n=0$ we obtain that $\lim_{n\to+\infty} q_n=+\infty.$ Consequently, for $n$ big enough one has
$$b_{n-1}=\frac{(q_{n-1}-s)((1-s\e_{n-1})^2q_{n-1}-2s)}{(1-s\e_{n-1})(1-s\e_{n})q_{n-1}q_n}\le \frac{(1-s\e_{n-1})q_{n-1}}{(1-s\e_{n})q_n}<1$$
and
$$\lim_{n\to+\infty}\frac{c_{n}}{\sqrt{\e_n}}=\lim_{n\to+\infty}\frac{2s}{q_n\sqrt{\e_n}}\frac{1}{(1-s\e_{n-1})(1-s\e_{n})^2}\left(\frac{s}{q_{n-1}}-\frac{s^2\e_n}{q_{n-1}}-s(\e_{n-1}-\e_n)\right)=0.$$

Consequently, by using the fact that $(x_n)$ is bounded we have
$$\lim_{n\to+\infty}\frac{ \left(\frac{1}{(1-\frac{s}{q_{n-1}})(1-s\e_n)}-\frac{(1-s\e_n)q_n}{2s(1-\frac{s}{q_{n-1}})}c_n\right)x_n-x^*}{q_n\sqrt{\e_n}}=0,$$
which combined with the fact that $\lim_{n\to+\infty}\frac{\eta_n-x^*}{q_n\sqrt{\e_n}}=0$ yields
$$\lim_{n\to+\infty}\frac{(1-s\e_n)}{2s(1-\frac{s}{q_{n-1}})}b_{n-1} \frac{x_n - x_{n-1}}{\sqrt{\e_n}}=0.$$
But $(b_n)$ is bounded and according to our hypothese cannot go to 0 as $n\to+\infty$, hence $\|x_n-x_{n-1}\|=o(\sqrt{\e_n})$ as $n\to+\infty.$

From here we deduce at once that $\|y_n-x_n\|=o(\sqrt{\e_n})$ as $n\to+\infty,$ hence in particular
$$\lim_{n\to+\infty}y_n=x^*.$$

From \eqref{cckek} we have
$\left(1-\frac{s}{q_{n-1}}\right)\eta_n-\eta_{n+1}=\frac{(1-s\e_n)q_n}{2}\n f_n(y_n)$ and using the fact that  $\| \eta_n\| \ = \ o\left( q_n \sqrt{\e_n} \right) \mbox{ as } n \to +\infty$ we obtain that
$\|\n f_n(y_n)\|=o(\sqrt{\e_n})$ as $n\to+\infty,$ and further that $\|\n f(y_n)\|=o(\sqrt{\e_n})$ as $n\to+\infty.$ By using the $L-$Lipschitz continuity of $\n f$ we get $\|\n f(x_n)-\n f(y_n)\|\le L\|y_n-x_n\|$ which combined with the facts that $\|y_n-x_n\|=o(\sqrt{\e_n})$ as $n\to+\infty$ and $\|\n f(y_n)\|=o(\sqrt{\e_n})$ as $n\to+\infty$ lead to $\|\n f(x_n)\|=o(\sqrt{\e_n})$ as $n\to+\infty$.

Finally, by using Lemma \ref{descent} we get
$$f(y_n)-\min f\le f(x_n)-\min f+\|\n f(x_n)\|\|y_n-x_n\|+\frac{L}{2}\|y_n-x_n\|^2$$
hence, from the facts that $f(x_n)-\min f=\mathcal{O}(\e_n)$, $\|y_n-x_n\|=o(\sqrt{\e_n})$ and $\|\n f(x_n)\|=o(\sqrt{\e_n})$ s $n\to+\infty$ we obtain that $$f(y_n)-\min f=\mathcal{O}(\e_n)\mbox{ as }n\to+\infty.$$
\end{proof}

\section{Particular choice of the parameter sequences $(q_k)$ and $(\e_k)$}

Let us consider a specific choice of the sequences $(q_k)$ and $(\e_k)$ being polynomial type, namely, $q_k = a k^q$, $\e_k = \frac{c}{k^p}$, where $1 \geq q > 0$, $p > 0$ and $a$ and $c$ are positive real numbers.
Let us fix $0<s<\frac{1}{L}.$ Then from condition (S) we have $s\le\frac{1}{L+\e_{k_0}}$ for some $k_0\in \N$, and it is an easy computation that
$$k_0=\mbox{int}\left(\frac{cs}{1-Ls}\right)^{\frac{1}{p}}+1,$$
where $\mbox{int}(x)$ denotes the integer part of $x.$

Now we  compute the index $k_1\ge k_0$  such that $1-s\e_k>0$ for all $k\ge k_1.$ Note that $1-s\e_k>0$ whenever $k\ge \mbox{int}(cs)^{\frac{1}{p}}+1$, consequently one can take $k_1=k_0.$

Condition $(Q)$ in this case becomes: after an index $k_2\ge k_1$ it holds that
$$\left( 1 - \frac{sc}{(k+1)^p} \right)^2 a^2 (k+1)^{2q} - \left( 1 - \frac{sc}{k^p} \right)^2 a^2 k^{2q} - 2 s a (k+1)^q + s \left( 1 - \frac{s c}{k^p} \right)^2 a k^q \ \leq \ 0$$
and
$$a k^q \ \geq \ \frac{2s}{(1 - \frac{s c}{k^p})^2}$$
for all $k \geq k_2$. Note that the second condition if always fulfilled starting from $k$ large enough due to $q$ and $p$ being positive.

Now consider the case $q < 1$. Since
$$\left( 1 - \frac{sc}{(k+1)^p} \right)^2 a^2 (k+1)^{2q} - \left( 1 - \frac{sc}{k^p} \right)^2 a^2 k^{2q}=\mathcal{O}(k^{2q-1})\mbox{ as }k\to+\infty$$
and
$$- 2 s a (k+1)^q + s \left( 1 - \frac{s c}{k^p} \right)^2 a k^q =-sa k^q+\mathcal{O}(k^{q-1})+\mathcal{O}(k^{q-p})\mbox{ as }k\to+\infty$$
we obtain that
\begin{align*}
&\left( 1 - \frac{sc}{(k+1)^p} \right)^2 a^2 (k+1)^{2q} - \left( 1 - \frac{sc}{k^p} \right)^2 a^2 k^{2q} - 2 s a (k+1)^q + s \left( 1 - \frac{s c}{k^p} \right)^2 a k^q =-sa k^q\\
&+\mathcal{O}(k^{2q-1})+\mathcal{O}(k^{q-1})+\mathcal{O}(k^{q-p})\mbox{ as }k\to+\infty,
\end{align*}
hence there exists an index $k_2\ge k_1$ such that (Q) holds.

 Now, if $q = 1$ we obtain
 \begin{align*}
 &\left( 1 - \frac{sc}{(k+1)^p} \right)^2 a^2 (k+1)^{2} - \left( 1 - \frac{sc}{k^p} \right)^2 a^2 k^{2} - 2 s a (k+1) + s \left( 1 - \frac{s c}{k^p} \right)^2 a k=(2a^2-a s)k\\
 & +\mathcal{O}(k^{1-p})+\mathcal{O}(1)\mbox{ as }k\to+\infty,
 \end{align*}
 hence (Q) holds provided $2a^2-a s<0$, that is $a<\frac{s}{2}.$

Concerning the sequence $(b_k)_{k\ge 0}$ in this particular case condition (B) becomes:
$$(Bp)\,\,\,\left\{\begin{array}{lll}
  b_{k-1}=0, \mbox{ if } k\in\left\{1,(cs)^{\frac1p},(cs)^{\frac1p}+1\right\}\\
  \\
 \ds b_{k-1}=\frac{k^p(a(k-1)^q-s)(a((k-1)^p-cs)^2(k-1)^q-2s(k-1)^{2p})}{a^2(k-1)^{q+p}k^{q}((k-1)^p-cs)(k^p-cs)},\mbox{ otherwise}.
\end{array}\right.$$
Note that $b_k\to 1$ as $k\to+\infty.$

Further, condition (C) becomes:
 $$(Cp)\,\,\,\left\{\begin{array}{lll}
  c_{k}=0, \mbox{ if } k\in\left\{1,(cs)^{\frac1p},(cs)^{\frac1p}+1\right\}\\
  \\
 \ds c_{k}=\frac{2s^2k^p((k-1)^pk^p-c(k-1)^p-ac(k-1)^qk^p+ac(k-1)^{q+p})}{a^2 (k-1)^qk^q((k-1)^p-cs)(k^p-cs)^2},\mbox{ otherwise}.
\end{array}\right.$$

Note that $c_k>0$ for $k$ big enough, further $(c_k)$ is  nonincreasing and $c_k\to0$ as $k\to+\infty.$ Hence, indeed in this case the term $c_k x_k$ in Algorithm \eqref{tdiscgen} plays the role of a Tikhonov regularization term. More precisely, Algorithm \eqref{tdiscgen} reads as: $x_0,x_{1}\in\mathcal{H}$ and for all $k\ge 1$
\begin{equation}\label{tdiscgenp}
\left\{\begin{array}{llll}
\ds y_k=x_k, \mbox{ if } k\in\left\{1,(cs)^{\frac1p},(cs)^{\frac1p}+1\right\}\\
\ds y_k= x_k+\frac{k^p(a(k-1)^q-s)(a((k-1)^p-cs)^2(k-1)^q-2s(k-1)^{2p})}{a^2(k-1)^{q+p}k^{q}((k-1)^p-cs)(k^p-cs)}(x_k-x_{k-1})\\
\ds\,\,\,\,\,\,\,\,\,\,-\frac{2s^2k^p((k-1)^pk^p-c(k-1)^p-ac(k-1)^qk^p+ac(k-1)^{q+p})}{a^2 (k-1)^qk^q((k-1)^p-cs)(k^p-cs)^2} x_k,\mbox{ otherwise}
\\
x_{k+1}=y_k-s\n f(y_k)-\frac{cs}{k^p} y_k.
\end{array}\right.
\end{equation}

Note that from a numerical point of view Algorithm \eqref{tdiscgenp} can easily be implemented. In this particular case, we have the following result.
\begin{theorem}\label{strongconvergencep}
Let  $0<q < 1$ and $0 < p < 2q$ and for a fixed the stepsize $s\in\left(0,\frac{1}{L}\right)$
 consider the sequences  $(x_k)_{k\in\N},\,(y_k)_{k\in\N}$  generated by Algorithm \eqref{tdiscgenp}. Then, $(x_k)$ and $(y_k)$ converge strongly to $x^*$, where $\{x^*\}=\pr_{\argmin f}(0)$ is the minimum norm minimizer of our objective function $f.$

Further,
$$f_k(x_{k})-f_{k}(\ol x_{k})=o(k^{-p})\mbox{ as }k\to+\infty,$$
$$f(x_k)-\min f=\mathcal{O}\left(k^{-p}\right),\mbox{ as }k\to+\infty\mbox{ and } f(y_k)-\min f=\mathcal{O}\left(k^{-p}\right),\mbox{ as }k\to+\infty,$$
$$\| x_k - x_{k-1} \| \ = \ o\left( k^{-\frac{p}{2}} \right) \mbox{ as } k \to +\infty$$
and
$$\| \n f(x_k)\| \ = \ o\left( k^{-\frac{p}{2}} \right) \mbox{ as } k \to +\infty\mbox{ and }\| \n f(y_k)\| \ = \ o\left( k^{-\frac{p}{2}} \right) \mbox{ as } k \to +\infty$$

\end{theorem}
\begin{proof} We only need to show that the following conditions from the hypotheses of Theorem \ref{strongconvergence} hold:
\begin{itemize}
    \item the sequence $\frac{q_k\e_k}{q_{k-1}\e_{k-1}}=\frac{k^{q-p}}{(k-1)^{q-p}}$ is bounded if we take the starting index big enough;

    \item the sequence $q_k^2\e_k=a^2 c k^{2q-p}$ is increasing after a starting index big enough;

    \item $\lim_{k\to+\infty}q_k^2\e_k=\lim_{k \to +\infty} a^2 c k^{2q-p} \ = \ +\infty$;

    \item $\lim_{k\to+\infty}\frac{ q_k(\e_k-\e_{k+1})}{\e_k}=\lim_{k \to +\infty} \frac{a k^q \left( \frac{1}{k^p} - \frac{1}{(k+1)^p} \right)}{\frac{1}{k^p}} \ = \ 0$.
\end{itemize}

First of all, the sequence $\frac{k^{q-p}}{(k-1)^{q-p}}$ is indeed bounded if we take the starting index big enough since $\lim_{k \to +\infty} \frac{k^{q-p}}{(k-1)^{q-p}} = 1$. Secondly, the sequence $a^2 c k^{2q-p}$ is increasing when $2q > p$ and $\lim_{k \to +\infty} a^2 c k^{2q-p} \ = \ +\infty$ when $2q > p$. Finally,
$$\lim_{k\to+\infty}\frac{a k^q \left( \frac{1}{k^p} - \frac{1}{(k+1)^p} \right)}{\frac{1}{k^p}} \  = \lim_{k\to+\infty}\frac{ak^q}{k+1} \frac{(k+1)^p - k^p}{(k+1)^{p-1}} =0,$$
since $q<1.$
\end{proof}

\begin{remark} We emphasize that Algorithm \eqref{tdiscgenp} can be seen as a Nesterov type algorithm with two Tikhonov regularization terms.
Indeed, the extrapolation parameter $(b_k)$ goes to 1 as $k\to+\infty$ such as in the Nesterov algorithm. Further, the terms $c_kx_k$ and $\e_k y_k$ can be thought as Tikhonov regularization terms since both $c_k$ and $\e_k$ are nonnegative and nonincreasing sequences (after $k$ big enough), and goes to 0 as $k\to+\infty.$ Unfortunately we could not allow the case $q=1$ and $p=2$ in our algorithm. Nevertheless, if $p$ is close to 2, (and $q$ is close to 1), then from a numerical perspective the convergence rates obtained for the potential energy $f(x_k)-\min f$ and discrete velocity $\|x_k-x_{k-1}\|$ are as good as the rates obtained for the famous Nesterov algorithm, see \cite{AP}. Moreover, our algorithm assures the strong convergence of the generated sequences to the minimum norm minimizer a feature that makes it unique in the literature.
\end{remark}

\section{Numerical experiments}
In this section we consider some numerical experiments in order to sustain the theoretical results obtained in Theorem \ref{strongconvergencep}. To this purpose, let us consider the  objective function
\[
f: \mathbb{R}^2 \mapsto \mathbb{R}, \ f(x, y) = (ax + by)^2,
\]
where $a, b \in \mathbb{R} \setminus \{0\}$. Then obviously $f$ is smooth and convex and its gradient is Lipschitz continuous, having  Lipschitz constant $L=2\sqrt{2}\sqrt{(a^2+b^2)\max(a^2,b^2)}.$
Observe that the minimal value of $f$ is 0 and the set $\argmin f$ is $\left\{\left(x, -\frac{a}{b}x \right):x\in\R\right\}$, further clearly $(0, 0)$ is the minimizer of  minimal norm. For simplicity in the following experiments concerning Algorithm \eqref{tdiscgenp} we take everywhere $q_k = k^\frac{4}{5}$ and $s = 0.1$, (which always will satisfy $s<\frac{1}{L}$), and fix the starting points $x_0=(1, -1)$ and $x_1=(-1, 1).$

In our first experiment we fix $a = 0.1$ and $b = 100$ and in Algorithm \eqref{tdiscgenp} we set  $\e_k = \frac{1}{k^p}$, where $p\in\{0.3,0.6,0.9,1.2,1.5\}$. Further, we consider the case when there is no Tikhonov regularization, that is the case of a Nesterov type algorithm by taking $\e_k\equiv 0$ and $c_k\equiv 0$. In order to show the theoretical rates obtained in Theorem \ref{strongconvergencep} for the discrete velocity $\|x_k-x_{k-1}\|$ and the potential energy $f(x_k)-\min f$  we also represent the values $1/k$ and $f(x_1)/k^2$, respectively.

So we run Algorithm \eqref{tdiscgenp} for 20 iterations, the results are shown in Figure 1.

\begin{figure}[hbt!]
    \includegraphics[width=\textwidth]{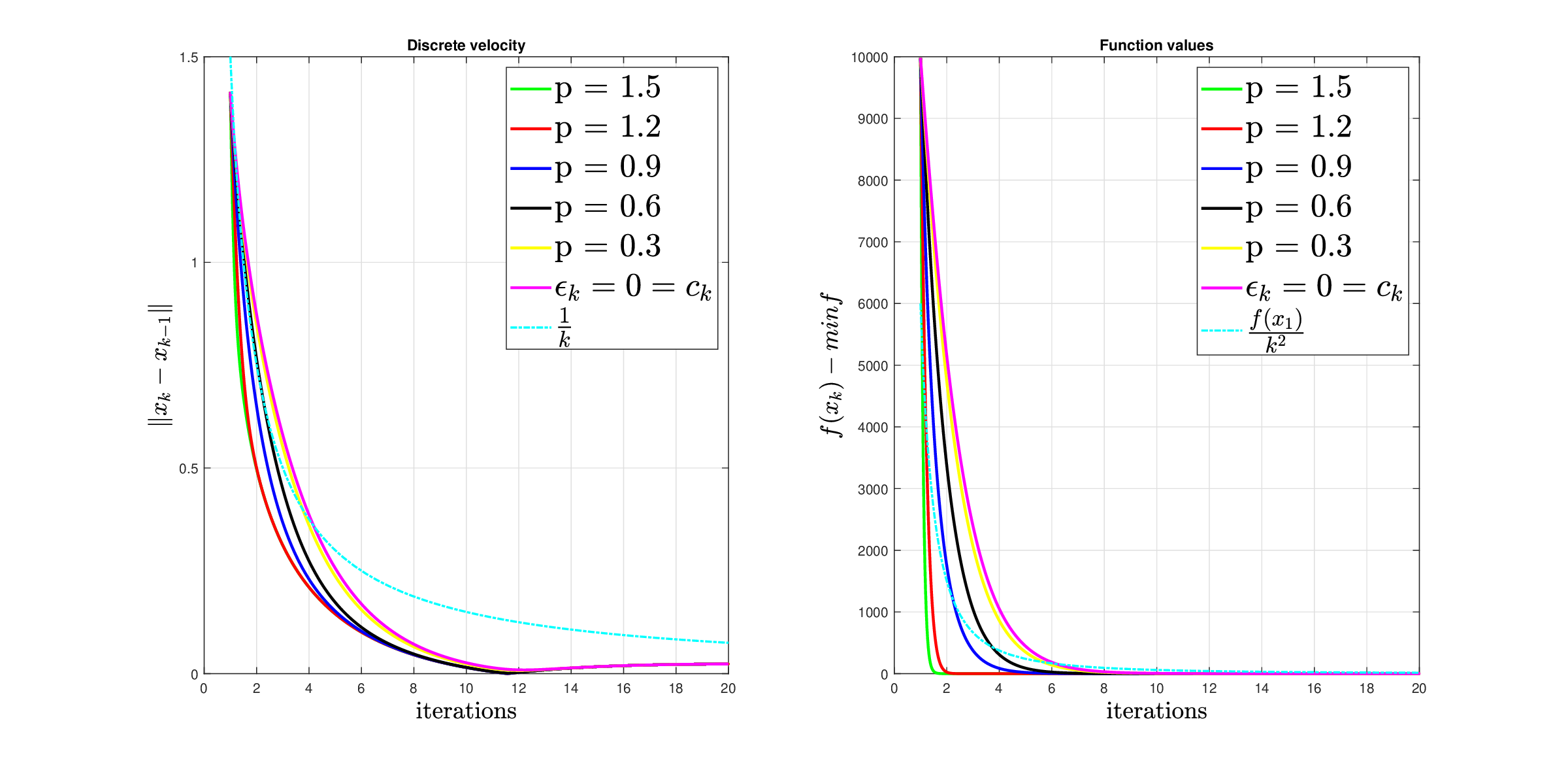}
    \caption{Different choices of $\e_k$}
\end{figure}

Observe that indeed, our algorithm has a similar (even better) behavior as the Nesterov type algorithm, the convergence rates for the discrete velocity and potential energy are of order $o(1/k)$ and $O(1/k^2)$, respectively. Further, the Tikhonov regularization does not affect the optimal rates, even more, while we increase $p$ these rates become better.

In our second experiment for $a=1,\,b=5$ we show the influence of the Tikhonov regularization terms $\e_ky_k$ and $c_kx_k$ on the behaviour of the iterates of the algorithm. In the next figures we represent the first component of the iterates $x_k$ with red meanwhile the second component will be represented with blue.

First,  we analyze what happens if we renounce to both Tikhonov regularization terms. So let us put both $c_k \equiv 0$ and $\e_k \equiv 0$  in Algorithm \eqref{tdiscgenp}. According to Figure 2 in this case there is  no convergence to the minimal norm element.

\begin{figure}[hbt!]
   \begin{center} \includegraphics[width=0.5\textwidth]{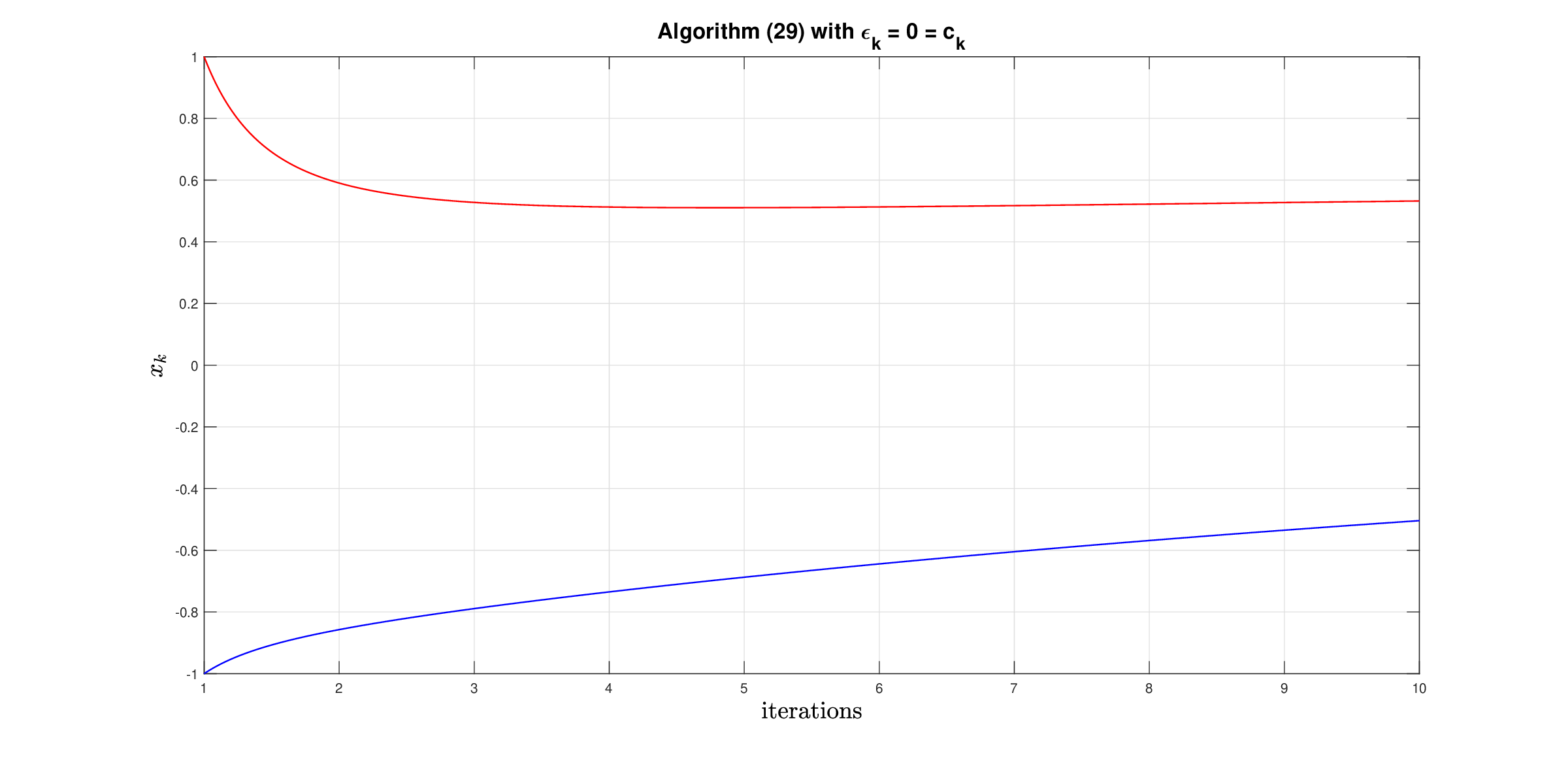}
    \caption{$\e_ky_k \equiv 0,\,c_kx_k\equiv 0$}
    \end{center}
\end{figure}

Next we show that in order to have convergence to the minimal norm element the presence of both Tikhonov regularization terms are essential. To this purpose, we take $\e_k=1/k^\frac{3}{2}$, (and the corresponding $c_k$), in order to show convergence to the minimum  norm minimizer and also $\e_k\equiv 0$ to show that in this case our algorithm does not converge anymore to the minimum  norm minimizer, see Figure 3(a).  Note that in case $\e_k\equiv 0$ the parameter $c_k$ in Algorithm \eqref{tdiscgenp} becomes $c_k=\frac{2s^2}{(k-1)^qk^q}$, hence the term $c_kx_k$ in the formulation of $y_k$ still has the role of a Tikhonov regularization term. Further, we consider the case $c_k \equiv 0$, but $\e_k=1/k^\frac{3}{2}$, see Figure 3(b).

\begin{figure}[hbt!]
\begin{subfigure}{.5\textwidth}
  \centering
  \includegraphics[width=.99\linewidth]{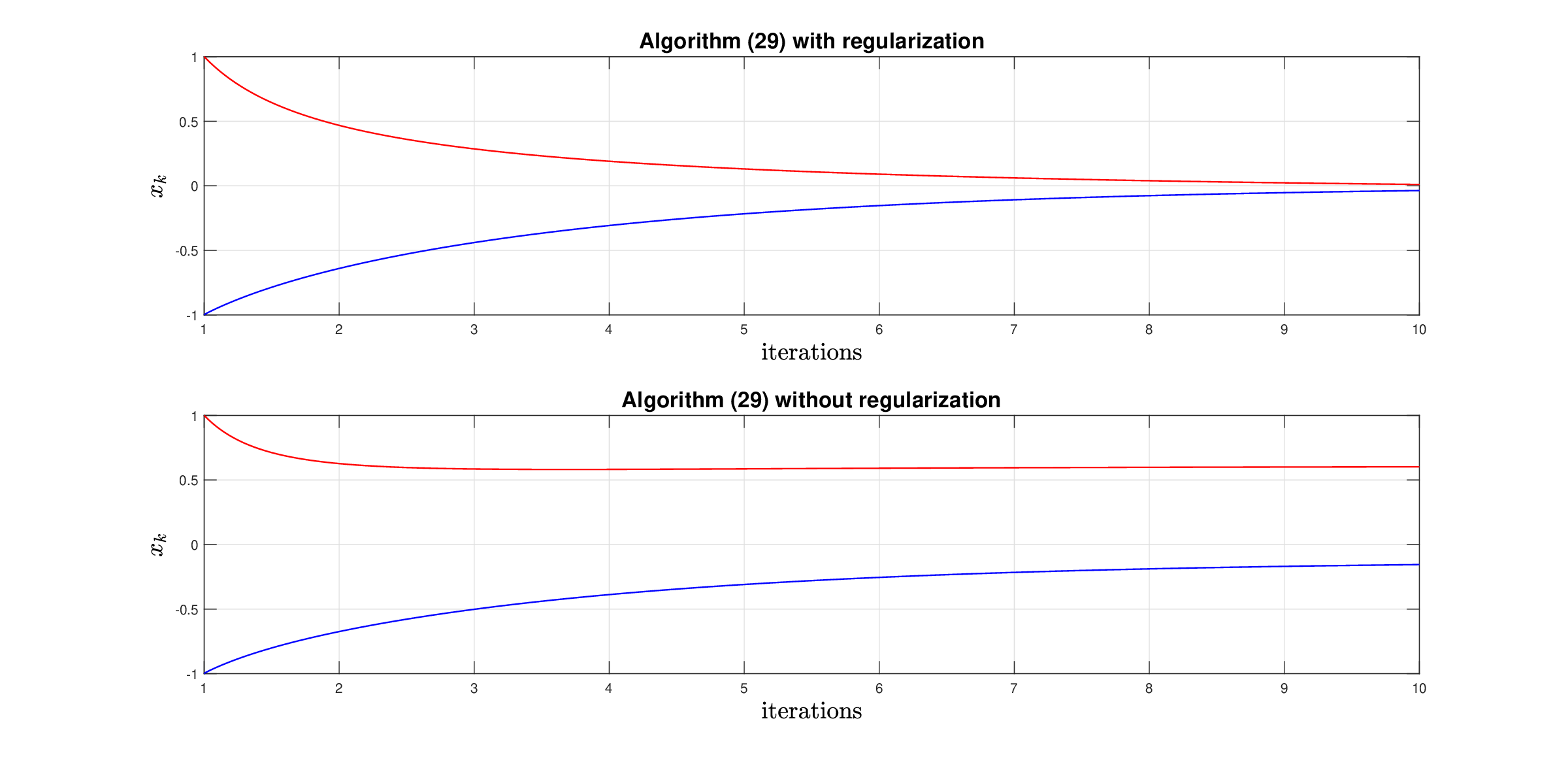}
  \caption{The absence of the term $\e_ky_k$}
  \label{fig2:sfig21}
\end{subfigure}
\begin{subfigure}{.5\textwidth}
  \centering
  \includegraphics[width=.99\linewidth]{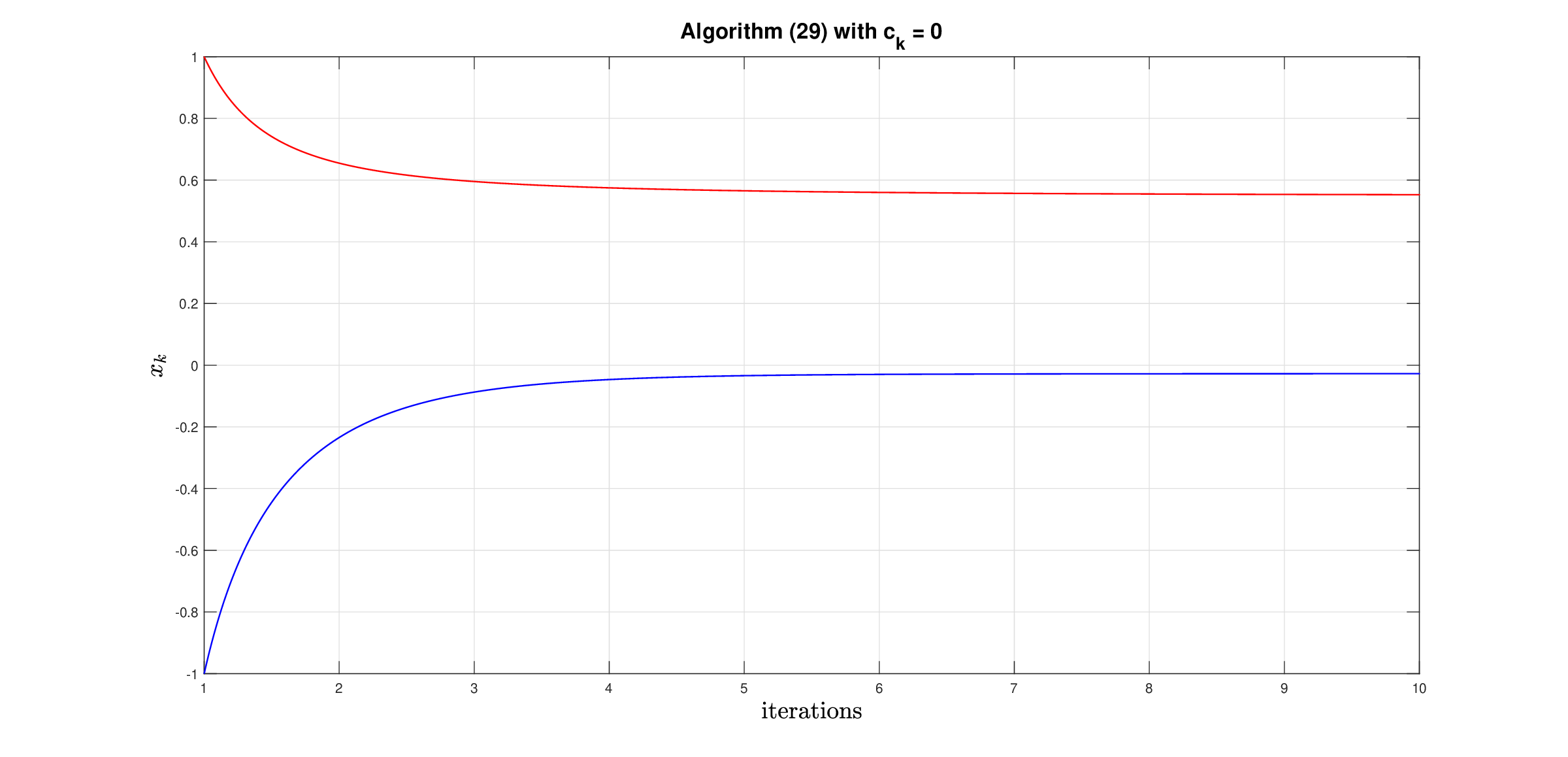}
  \caption{ The absence of the term $c_kx_k$}
  \label{fig2:sfig22}
\end{subfigure}
 \caption{Dropping one of the Tikhonov regularization terms in Algorithm \eqref{tdiscgenp} we do not have convergence to the minimum norm solution anymore.}
\end{figure}

As we can see, in the absence of one of the Tikhonov regularization terms  we do not have  the convergence to the element of the minimal norm. Hence,  according to the last two figures  the presence of double Tikhonov regularization terms in our algorithm is fully justified.


\section{Conclusions, perspectives}

Due to our best knowledge, Algorithm \eqref{tdiscgen} and in particular Algorithm \eqref{tdiscgenp} are the first inertial gradient type algorithms considered in the literature that assure strong convergence to the minimum norm minimizer of a smooth convex function and also fast convergence of the function values and  discrete velocity. As we have emphasized in the paper these algorithms can be seen as Nesterov type algorithms with two Tikhonov regularization terms. Despite of the complex structure of the inertial parameter and  one of the Tikhonov regularization parameters our algorithms can easily be implemented, therefore are suitable for use in practical problems arising in image processing and machine learning. As a future related research we mention here the forward-backward algorithms with Tikhonov regularization associated to the minimization problem having in its objective the sum of a proper convex lower semicontinuous function and a smooth convex function with Lipschitz continuous gradient. In our opinion similar results to those provided in Theorem \ref{strongconvergencep} can be obtained. Indeed, the success of such research is promising taking into account that in \cite{L-mapr} strong convergence of an inertial-proximal algorithm to the minimal norm minimizer of a proper convex and lower semicontinuous function is shown, meanwhile in the present paper we obtained similar results for inertial gradient type algorithm in connection to a smooth convex optimization problem.

\section{Declarations}

\begin{center}
    \textbf{Availability of data and materials}
\end{center}

In this manuscript only the datasets generated by authors were analysed.

\begin{center}
    \textbf{Competing interests}
\end{center}

The authors have no competing interests.

\end{document}